\documentclass{llncs}
\usepackage{amsmath,amsfonts,amsthm,amssymb,amscd}
\usepackage[latin1]{inputenc}
\usepackage{mathrsfs}
\usepackage{latexsym}
\usepackage{graphicx}
\usepackage{enumerate}
\usepackage{complexity}
\usepackage{hyperref}
\newtheorem{Exa}{Example}

\newcommand{\tnc}[1]{}
\newcommand{\tnm}[1]{\marginpar{\color{red}\tiny\textsf{#1}}}
\renewcommand{\tnm}[1]{}

\renewcommand{\ge}{\geqslant}

\usepackage{pgf}
\usepackage{tikz}
\usetikzlibrary{arrows}
\usetikzlibrary{decorations.pathmorphing}
\usetikzlibrary{backgrounds}
\usetikzlibrary{fit}

\begin{document}
\title{\mbox{A Cellular Automaton for Blocking Queen Games}}

\author{Matthew Cook\inst{1} and Urban Larsson\inst{2} and Turlough Neary\inst{1}}
\institute{Institute of Neuroinformatics, University of Z\"urich and ETH Z\"urich, Switzerland\\
 \and Department of Mathematics \& Statistics, Dalhousie University, Halifax, Canada\footnote{Supported by the Killam Trust.  Contact email: \email{urban031@gmail.com}}
}

\date{\today }
\maketitle

\begin{abstract}
We show that the winning positions of a certain type of two-player game
form interesting patterns which often defy analysis,
yet can be computed by a cellular automaton.
The game, known as {\em Blocking Wythoff Nim},
consists of moving a queen as in chess, but always towards~(0,0),
and it may not be moved to any of $k-1$ temporarily ``blocked'' positions
specified on the previous turn by the other player.
The game ends when a player wins by blocking all possible moves of the other player.
The value of~$k$ is a parameter that defines the game,
and the pattern of winning positions can be very sensitive to~$k$.
As~$k$ becomes large,
parts of the pattern of winning positions
converge to recurring chaotic patterns that are independent of~$k$.
The patterns for large~$k$ display an unprecedented amount of self-organization
at many scales,
and here we attempt to describe the self-organized structure that appears.
\end{abstract}



\section{Blocking Queen Games ($k$-Blocking Wythoff Nim)}\label{sec:gamedef}

In the paper~\cite{Lar11},
the game of $k$-Blocking Wythoff Nim was introduced,
with rules as follows.

\begin{itemize}

\item[Formulation 1:]
As in Wythoff Nim~\cite{Wyt07},
two players alternate in removing counters from two heaps:
any number may be removed from just one of the heaps,
or the same number may be removed from both heaps.
However, a player is allowed to reject the opponent's move
(so the opponent must go back and choose a different, non-rejected move),
up to $k-1$ times, where $k$ is a parameter that is fixed for the game.
The $k^{\rm th}$ distinct attempted move must be allowed.
Thus, if there are at least $k$ winning moves
among the options from a given position,
then one of these winning moves can be played.

\item[Formulation 2:]
There are $k$ chess pieces on an infinite (single quadrant) chess board:
one queen, and $k-1$ pawns.
On your turn you move the queen towards the origin.
(The first player who cannot do this loses.)
The queen cannot be moved to a position with a pawn,
but it can move over pawns to an empty position.
After moving the queen,
you complete your turn by moving the $k-1$ pawns to wherever you like.
The pawns serve to block up to $k-1$ of the queen's possible next moves.

\item[Example Game:]
Consider a game with $k=5$,
where the queen is now at $(3,3)$ (yellow in Figure~\ref{fig:game5}).
It is player $A$'s turn,
and player $B$ is blocking the four positions $\{(0,0), (1,1), (0,3), (3,0)\}$
(dark brown and light olive).
This leaves $A$ with the options
$\{(3,1),(3,2),(2,2),(2,3),(1,3)\}$
(each is black or blue).
Regardless of which of these $A$ chooses,
$B$ will then have at least five winning moves to choose from
(ones marked yellow, or light, medium, or dark olive).
These are winning moves because it is possible when moving there
to block all possible moves of the other player and thereby immediately win.
Therefore player $B$ will win.
\end{itemize}

\begin{figure}[h]
\vspace{-4mm}
\centering{
\begin{tikzpicture}
\node[inner sep = 0pt, anchor = north west] (grid) at (0,0)
    {\includegraphics[width = 0.4\textwidth]{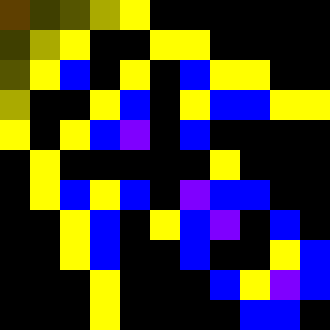}};
\newcommand*{\mygap}{0.447}
\newcommand*{\mytop}{0.05}
\newcommand*{\myleft}{0.05}
\foreach \i in {0,...,10}
{
    \pgfmathsetmacro{\x}{(\i + 0.5)} 
    \draw (\x * \mygap, \mytop) node[anchor=south] {$\i$};
    \pgfmathsetmacro{\y}{(\i + 0.5)}
    \draw (-1 * \myleft, \y * \mygap * -1) node[anchor=east] {$\i$};
}
\end{tikzpicture}
}
\caption{Game 5: 
$(0,0)$~is the upper left position on this $11\times 11$ chessboard,
and the queen is allowed to move north, west, or north-west.
The colors represent the number of winning moves available (ignoring blocking)
from each position.
0:~Dark Brown, 1:~Dark Olive, 2:~Olive, 3:~Light Olive, 4:~Yellow,
5:~Black, 6:~Blue, 7:~Indigo.
Since $k=5$, up to four moves can be blocked.
This means that by moving to a position with four or fewer winning moves available,
it is possible to block all those winning moves,
thus preventing the other player from making a winning move.
Therefore the positions with four or fewer winning moves available
(yellow, brown, and the olives)
are themselves winning moves.
The remaining positions, with five or more winning moves available
(black, blue, and indigo),
are losing moves,
because if you move there,
it is not possible to prevent the other player from making a winning move.
Thus the color at any given position
(the {\em palace number}, see text)
can be computed
by considering the colors above and to the left of it.
}
\label{fig:game5}
\end{figure}
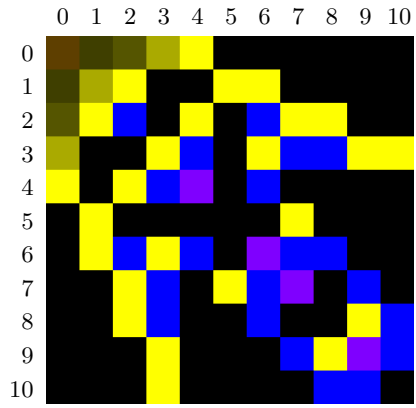

As shown in Figure~\ref{fig:game5},
there is a simple algorithm to compute the winning positions for game~$k$.
These are known as {\em P-positions} in combinatorial game theory,
and we will refer to them as {\em palace positions},
the idea being that the queen wants to move to a palace:
if you move her to a palace, you can win,
while if you move her to a non-palace, your opponent can win.
To win,
you must always block (with pawns) all of the palaces your opponent might move to.

The idea is simply that
a palace is built on any site
that can see fewer than $k$ other palaces
when looking due north, west, or north-west.
In this way, the pattern of palaces can be constructed,
starting at (0,0) and going outward.
For efficiency, a dynamic programming approach can be used,
storing three numbers at each position,
for the number of palaces visible in each of the three directions.
With this technique,
each line can be computed just from the information in the previous line,
allowing significant savings in memory usage.

The case $k=1$ corresponds to classical Wythoff Nim,
solved in~\cite{Wyt07}.
In~\cite{Lar11}, the game was solved for $k = 2$ and $3$.
When we say a game is solved we mean it is possible to give
closed-form expressions for the P-positions,
or at least that a winning move, if it exists,
can be found in log-polynomial time in the heap sizes.
For example, the set
$\{(\lfloor n\phi \rfloor, \lfloor n\phi^2 \rfloor ),
(\lfloor n\phi^2 \rfloor, \lfloor n\phi \rfloor )\}$,
where $n$ runs over the nonnegative integers and
$\phi=\frac{1+\sqrt{5}}{2}$ is the golden ratio,
provides a solution for classical Wythoff Nim. Combinatorial games with a blocking maneuver appears in \cite{GaSt04}, \cite{HoRe01}, \cite{HoRe} and \cite{SmSt02}, and specifically for Wythoff Nim in \cite{Gur10}, \cite{HeLa06}, \cite{Lar09} and \cite{Lar15}.

The new idea that we use in this paper is to look directly at
the number of palaces that
the queen can see from each position on the game board,
and to focus on this {\em palace number}
rather than just on the palace positions (P-positions)\@.
(The palace positions are exactly the locations with palace numbers less than $k$.)
In previous explorations, the palace number was in fact computed
but then only used for finding the new palace positions,
which when viewed on their own
give the appearance of involving long-range information transfer.
By observing the palace number directly, however,
we can see that in almost all positions
the palace number is surprisingly close to $k$,
and if we look at its deviation from $k$
then we are led to discover that these deviations follow
a local rule that does not even depend on $k$.
The next section will present this local rule as a cellular automaton (CA). A finite automaton for Wythoff Nim has been studied in a different context in \cite{Lan02}. Other recent results for cellular automata and combinatorial games can be found in \cite{Fin12}, \cite{Lar13} and  \cite{LaWa13}.

The rest of the paper will present the rich structure visible
in the patterns of palace numbers for games with large $k$.
A surprising number regions of self organization appear,
largely independent of the particular value of $k$,
and some of the self-organized patterns are quite complex,
involving multiple layers of self-organization.
This is the first CA we are aware of that exhibits so many levels of self-organization.
So far, these patterns offers many more questions than answers,
so for now we will simply try to catalog our initial observations.


\section{A Cellular Automaton Perspective}

In this section we describe a cellular automaton 
that computes the palace numbers for blocking queen games.

\subsection{Definition of the CA}

The CA we present is one-dimensional
and operates on a diamond space-time grid
as shown in Figure~\ref{fig:carule},
so that at each time step,
each cell that is present at that time step derives from two parents
at the previous time step,
one being half a unit to the right,
the other being half a unit to the left.
On the diamond grid, there is no cell from the previous step
that is `the same' as a given cell on the current step.
However, the three grandparents of a cell,
from two steps ago,
do include a cell which is
at the same spatial position as the grandchild cell.

Our CA rule is based not only on the states of the two parent cells,
but also on the state of the central grandparent cell,
as well as {\em its} parents and central grandparent,
all shown in blue in Figure~\ref{fig:carule}.
As such,
this CA depends on the previous four time steps,
{\it i.e.}~it is a fourth-order CA.
It is the first naturally occurring fourth-order CA that we are aware of.
We say it is ``naturally occurring'' simply because we discovered
these pictures by analyzing the blocking queen game,
and only later realized that these pictures
can also be computed by the fourth-order diamond-grid CA
we present here.%
\footnote{
Using the dynamic programming approach described in Section~\ref{sec:gamedef},
the information for each cell can be computed
from just its parents and central grandparent,
so that approach is also clearly a CA,
but it needs a large number of states for the cells
and each different value of $k$ requires a different rule.
The CA we describe in this section is much more surprising,
being a single rule that is completely independent of $k$
and using very few states to produce all of the interesting parts
of the pictures.
}

The states in our CA are integers.
In general they are close to 0,
but the exact bounds depend on the initial conditions.
The formula for computing a cell's value,
given its neighborhood,
is described in Figure~\ref{fig:carule}.

\begin{figure} [h]
\begin{center}
\hspace{-20pt}%
\begin{tikzpicture} [scale = 0.5]

\foreach \x/\y in {-11/0, -12/0, -13/0, -14/0}
 \draw[thin, gray] (\x, \y) -- (\x+5, \y+5);
\foreach \x/\y in {-9/0, -8/0, -7/0, -6/0}
 \draw[thin, gray] (\x-1, \y) -- (-6, \y-5-\x);
\foreach \x/\y in {-14/1, -14/2, -14/3, -14/4}
 \draw[thin, gray] (\x, \y) -- (\x-\y+5, 5);
\foreach \x/\y in {-11/0, -12/0, -13/0, -14/0}
 \draw[thin, gray] (\x, \y+5) -- (\x+5, \y);
\foreach \x/\y in {-9/0, -8/0, -7/0, -6/0}
 \draw[thin, gray] (\x-1, \y+5) -- (-6, \y+10+\x);
\foreach \x/\y in {-14/1, -14/2, -14/3, -14/4}
 \draw[thin, gray] (\x, \y) -- (\x+\y, 0);
\draw[thick, blue] (1 + -12, 4) -- (1 + -11, 3);
\draw[thick, blue] (1 + -12.5, 3.5) -- (1 + -11, 2);
\draw[thick, blue] (1 + -13, 3) -- (1 + -12, 2);
\draw[thick, red] (1 + -12, 2) -- (1 + -11.5, 1.5);
\draw[thick, blue] (1 + -13, 2) -- (1 + -12.5, 1.5);
\draw[thick, red] (1 + -12.5, 1.5) -- (1 + -12, 1);
\draw[thick, blue] (-23 - -12, 4) -- (-23 - -11, 3);
\draw[thick, blue] (-23 - -12.5, 3.5) -- (-23 - -11, 2);
\draw[thick, blue] (-23 - -13, 3) -- (-23 - -12, 2);
\draw[thick, red] (-23 - -12, 2) -- (-23 - -11.5, 1.5);
\draw[thick, blue] (-23 - -13, 2) -- (-23 - -12.5, 1.5);
\draw[thick, red] (-23 - -12.5, 1.5) -- (-23 - -12, 1);
\draw (-10,5.5) node {$\leftarrow \mbox{space} \rightarrow$};
\draw (-14.5,2.5) node[rotate=90] {$\leftarrow \mbox{time}$};

\foreach \x/\y in {1/0, 0/1, 1/2, 2/1, 0/2}
 \draw[thick, blue] (\x, \y) rectangle (\x+1, \y+1);

\foreach \x/\y in {2/0}
 \draw[thick, red] (\x, \y) rectangle (\x+1, \y+1);

\draw (0.5 , 2.5) node {$a$};
\draw (1.5 , 2.5) node {$b$};
\draw (0.5 , 1.5) node {$c$};
\draw (1.5 , 1.5) node {$d$};
\draw (2.5 , 1.5) node {$e$};
\draw (1.5 , 0.5) node {$f$};
\draw (2.5 , 0.5) node {$g$};

\draw[->] (-2,.5) -- (-0.5,.5); \draw (-2.5 , 2.5) node {$h_1$};
\draw[->] (-2,1.5) -- (-0.5,1.5); \draw (-2.5 , 1.5) node {$h_2$};
\draw[->] (-2,2.5) -- (-0.5,2.5); \draw (-2.5 , 0.5) node {$h_3$};

\draw[->] (.5,5) -- (0.5, 3.5); \draw (.5 , 5.5) node {$v_1$};
\draw[->] (1.5,5) -- (1.5, 3.5); \draw (1.5 , 5.5) node {$v_2$};
\draw[->] (2.5,5) -- (2.5, 3.5); \draw (2.5 , 5.5) node {$v_3$};

\draw[->] (-1,5) -- (0, 4); \draw (-1.5 , 5.5) node {$d_1$};
\draw[->] (-1.5,4.5) -- (-.5, 3.5); \draw (-2 , 5) node {$d_2$};
\draw[->] (-2,4) -- (-1, 3); \draw (-2.5 , 4.5) node {$d_3$};

\foreach \x/\y in {6/0, 5/1, 6/2, 7/1, 5/2}
 \draw[thick, green]  (\x, \y) rectangle (\x+1, \y+1);

\draw (5.5 , 2.5) node {$3$};
\draw (6.5 , 2.5) node {$-2$};
\draw (5.5 , 1.5) node {$-2$};
\draw (6.5 , 1.5) node {$-1$};
\draw (7.5 , 1.5) node {$1$};
\draw (6.5 , 0.5) node {$1$};

\end{tikzpicture}
\end{center}

\caption{\label{fig:carule}
The CA rule.
{\bf (left)}
The diamond grid on which the CA operates.
The value of the red cell is determined by the values of the blue cells.
This neighborhood stretches four steps back in time
(i.e., four levels above the red cell),
making it a $4^{\rm th}$ order CA\@.
{\bf (center)}
This diagram is rotated $45^\circ$,
so time flows in the direction of the diagonal arrows.
This is the orientation used in all the figures in this paper.
The red cell's value is computed
according to the formula $g = a - b - c + e + f + p$
{\bf (right)}
These green squares correspond to the blue cells
and show the {\em palace compensation terms}.
For any blue cell containing a negative value
(and therefore a palace, see Section~\ref{sec:relation}),
the corresponding palace compensation term must be added.
In the formula, $p$ represents
the total contribution of these palace compensation terms.
Note that location $d$ only affects $g$ via its
palace compensation term, so only its sign matters.
}
\end{figure}
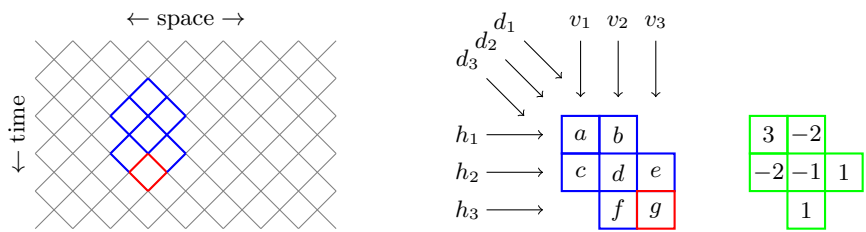



\subsection{The Connection between the CA and the Game}\label{sec:relation}

Since our definition of the CA appears to be completely different
from the definition of the blocking queen game,
we need to explain the correspondence.

The idea is that the states of this CA correspond to palace numbers minus~$k$, which are generally integers close to zero.
One can easily prove a bound of $3k+1$
for the number of states needed
for the game with blocking number $k$,
since each row, column, and diagonal
can have at most $k$ palaces,
so palace numbers are always in the range $[0,3k]$.
However,
in practice the number of states needed
(after the initial ramping-up region near the origin)
appears to be far smaller, more like $\log{k}$.
For example, when $k=500$,
only eight states are needed, ranging between -4 and 3.

Surprisingly,
this single CA is capable of computing the pattern of palace numbers
regardless of the value of $k$.
Different values of $k$ simply require different initial conditions:
The initial condition for a given value of $k$
is that every site in the quadrant opposite the game quadrant
should be $k$,
and every site in the other two quadrants should be 0.

\begin{theorem}\label{thm:ca}
The $k$-Blocking Wythoff Nim position $(x,y)$ is a P-position
if and only if the CA given in Figure~\ref{fig:carule}
gives a negative value at that position,
when the CA is started from an initial condition
defined by
$$CA(x,y)=\left\{\begin{array}{ccl}
k & ~~~ & x < 0 \mbox{~and~} y < 0 \\
0 & ~~~ & x < 0 \mbox{~and~} y\ge0 \\
0 & ~~~ & x\ge0 \mbox{~and~} y < 0 \\
\end{array}\right.$$
\end{theorem}

\begin{proof}
First we will consider the case of
no P-positions occurring within the CA neighborhood,
so compensation terms can be ignored.

As shown in Figure~\ref{fig:carule},
we will let $v_1$ be the number of P-positions directly above $a$ and $c$,
and similarly for $v_2$ and $v_3$,
as well as for the diagonals $d_i$ and horizontal rows $h_i$.

This gives us $a=v_1+d_2+h_1-k$, and so on:
when adding the $k$-value to each cell this represents the sum of the numbers of P-positions in the three directions.

We would like to express $g$ in terms of the other values.
Notice that
$$a+e+f \; = \;
\sum_{i=1}^3 v_i \; + \; \sum_{i=1}^3 d_i \; + \; \sum_{i=1}^3 h_i
\; - 3k = \; b+c+g$$
and therefore $a+e+f=b+c+g$,
allowing us to express $g$ in terms of the other values
as $g=a-b-c+e+f$.

All that remains is to take any P-positions in the CA neighborhood into account,
so as to understand the compensation terms.

If there is a P-position at $a$,
then $b$, $c$, $d$, and $g$
(i.e.\ the positions in a line to the right, down, or right-down)
will all be one higher than they were before taking that palace into account.
Since the equation $a+e+f=b+c+g$ was true when ignoring the palace at $a$,
it becomes wrong when the palace at $a$
produces its effect of incrementing $b$, $c$, $d$, and $g$, because that makes the right hand side go up by 3
while the left hand side is untouched.
To compensate for this, we can add a term $p_a$ to the left hand side,
which is 3 if there is a palace at $a$, and 0 otherwise.

Similarly, if $b$ is a P-position,
then this increments $d$, $e$, and $f$,
so to compensate, we will need to subtract 2
from the left hand side of $a+e+f=b+c+g$ if $b$ is a P-position.

We can see that we are computing exactly the compensation terms 
shown in the green squares of Figure~\ref{fig:carule}.
Once we include all the compensation terms,
the formula for $g$ becomes correct even in the presence of local P-positions,
and it corresponds exactly to the rule given in Figure~\ref{fig:carule}.

The initial condition can be confirmed to produce (via the CA rule) 
the correct values in the first two rows and columns of the game quadrant, 
and from that point onwards the reasoning given above shows that the 
correct palace numbers, and therefore the correct P-positions,
are being computed by the CA rule.
\end{proof}

\subsection{Notes on Reversability}

The reversed version of this CA
computes $a$, given $b$, $c$, $d$, $e$, $f$, and $g$.
This is done with the equation $a=g-f-e+c+b-p$,
which is equivalent to the equation
in the caption of Figure~\ref{fig:carule}.
However,
the palace compensation term $p$
can depend on $a$,
so this equation has not
fully isolated $a$ on the left hand side.
(The forward direction did not have this problem,
since $p$ does not depend on $g$.)
Writing $p=p_a+p_{bcdef}$ to separate
the palace compensation term from $a$
from the other palace compensation terms,
we get $a+p_a = g-f-e+c+b-p_{bcdef}$.
Since $p_a$ is 3 when $a$ is negative, and 0 otherwise,
this equation always yields either one or two solutions for $a$.
If the right hand side is 3 or more,
then $a$ must be equal to it.
If the right hand side is negative,
then $a$ must be 3 less than it.
And if the right hand side is 0, 1, or 2,
then $a$ can {\em either} be equal to it
{\em or} be 3 less than it---we are free to choose.
The reversed rule is non-deterministic,
but it can always find a compatible value.
In other words, there is no ``Garden of Eden'' pattern for this rule,
if we assume that all integers are permissible states.


\section{Self-Organization}\label{sec:selforg}

\begin{figure}[ht!]
\centering{
\includegraphics[width=0.45\textwidth]{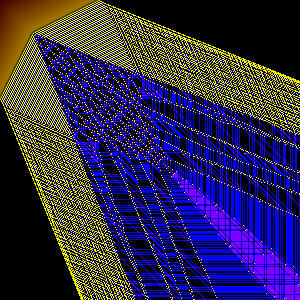}
~~~~~
\includegraphics[width=0.45\textwidth]{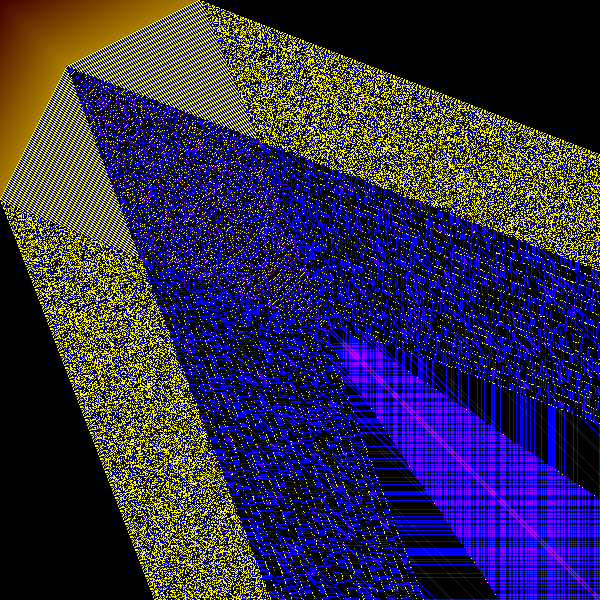}

\vspace{.15 cm}

\includegraphics[width=0.45\textwidth]{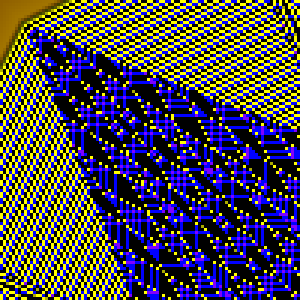}
~~~~~
\includegraphics[width=0.45\textwidth]{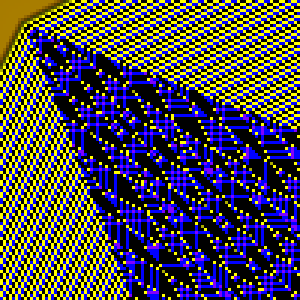}

\vspace{.15 cm}

\includegraphics[width=0.45\textwidth]{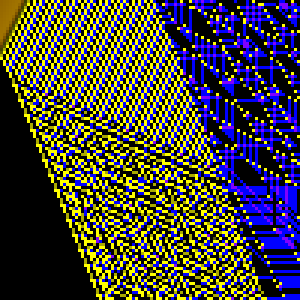}
~~~~~
\includegraphics[width=0.45\textwidth]{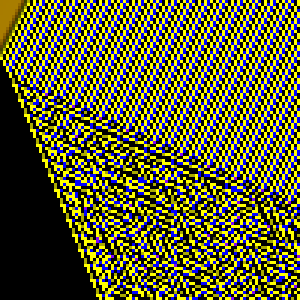}

}
\caption{\label{fig:selforg}\label{fig:prefix}
Self-organized regions in game 100 (left, $300 \times 300$ region shown in top row)
and game 1000 (right, $3000 \times 3000$ region shown in top row).
The lower images are close-ups of the upper images, showing regions of identical behavior.
The middle row shows a $100 \times 100$ region by the nose,
and the bottom row shows a $100 \times 100$ region by the shoulder.
\vspace{-20pt} 
}
\end{figure}

The top row of Figure~\ref{fig:selforg}
shows the palace number patterns for games 100 and 1000.
The patterns are strikingly similar,
given that the value of $k$ differs by an order of magnitude.
The pattern for game 1000 has the appearance of being
``the same, but ten times bigger'' than the pattern for game 100.
The middle and lower rows of Figure~\ref{fig:selforg}
zoom in on subregions where the two patterns are in fact identical,
without any scaling factor.

\subsection{Terminology}

As an aid to our discussion of these complex images,
we will give names to the prominent features in them.

We can see that this system self-organizes itself into
11 regions with 14 borders and 6 junctions.
Ignoring duplicates due to the mirror symmetry,
there are 7 regions, 7 borders, and 4 junction points.

We will start by naming the 7 regions and some features inside them.
The region at the upper left
(the game's terminal region, and the CA's starting region),
in the shape of a dented triangle, is the {\em hood}.
The triangular regions adjacent to the hood,
with a periodic interior
(visible in all panels of Figure~\ref{fig:selforg}),
are the {\em \'epaulettes}.
The rhomboid region that emanates from between the pair of \'epaulettes
is the {\em fabric}.
The solid black regions at the top and at the left
constitute the {\em outer space}.
Between the outer space and the \'epaulettes we find the {\em arms}
(irregular yellow regions in Figure~\ref{fig:selforg}),
which extend indefinitely.
Extending next to the arms, and of similar width, we have the {\em warps}.
Each warp contains a number of {\em threads}
(strings of yellow dots, clearly visible in the top left panel of
Figure~\ref{fig:prefix})
which come out of the fabric.
Between the warps lies the {\em inner sector},
and the blue stripes in the warps and in the inner sector are the {\em weft}.

Next, we will name the 4 junction points.
The hood, \'epaulettes, and fabric all meet at the {\em nose}.
The hood, \'epaulette, arm, and outer space all meet at the {\em shoulder}.
The warp, fabric, \'epaulette, and arm all meet at the {\em armpit},
which is often a hotspot of highly positive palace numbers.
And the fabric, warps, and inner sector meet at the {\em prism}.
The inner sector often contains slightly higher palace numbers than the warps,
especially near the main diagonal,
giving the impression of light being emitted from the prism,
as in Figure~\ref{fig:laser}.

Finally, we come to the 7 borders.
The hood and \'epaulette meet cleanly at the {\em casing}.
The hood contains all the states from $-k$ to $0$,
but after the casing, the CA uses very few states.
The \'epaulette and arm meet at the {\em hem}.
The \'epaulette and fabric meet at the {\em rift},
a narrow, relatively empty space which appears to get wider very slowly.
The fabric and warp meet at the {\em fray},
where threads almost parallel to the fray unravel from the fabric,
and threads in the other direction exit the fabric
and start merging to form thick threads in the warp.
There is no clear boundary where the warp meets the inner sector,
since the warp simply runs out of threads.
The warp also meets the arm cleanly, at the inside of the arm.
At the boundary between the warp and the arm,
it appears that the yellow nature of the arm is due to being packed full of threads,
and the warp simply has a much lower density of threads.
Threads that bend into the inner sector, and stop being parallel to the rest of the warp,
are sometimes called {\em beams} (see also \cite{Lar12}).
The often-occurring slightly-separated periodic part of the arm,
bordering the outer space, is the {\em skin}.

The fray, warp, central sector, armpits, and prism are all very sensitive to $k$,
but all the other regions are not,
with the exception of the fabric and the rift, which are sensitive only to $k \bmod 3$.
The fabric, fray, warp, prism, and central sector are all full of weft.
Often the centermost beams will communicate with each other via the weft,
and this process can usually be analyzed to calculate the slopes of these beams,
which are generally quadratic irrationals.

This is the greatest complexity of self-organization that we have seen for a system that
has no structured input.

\subsection{Structure Within the Regions}

\begin{figure}[ht!]
\centering{
\hspace{-10pt}\includegraphics[width=0.48\textwidth]{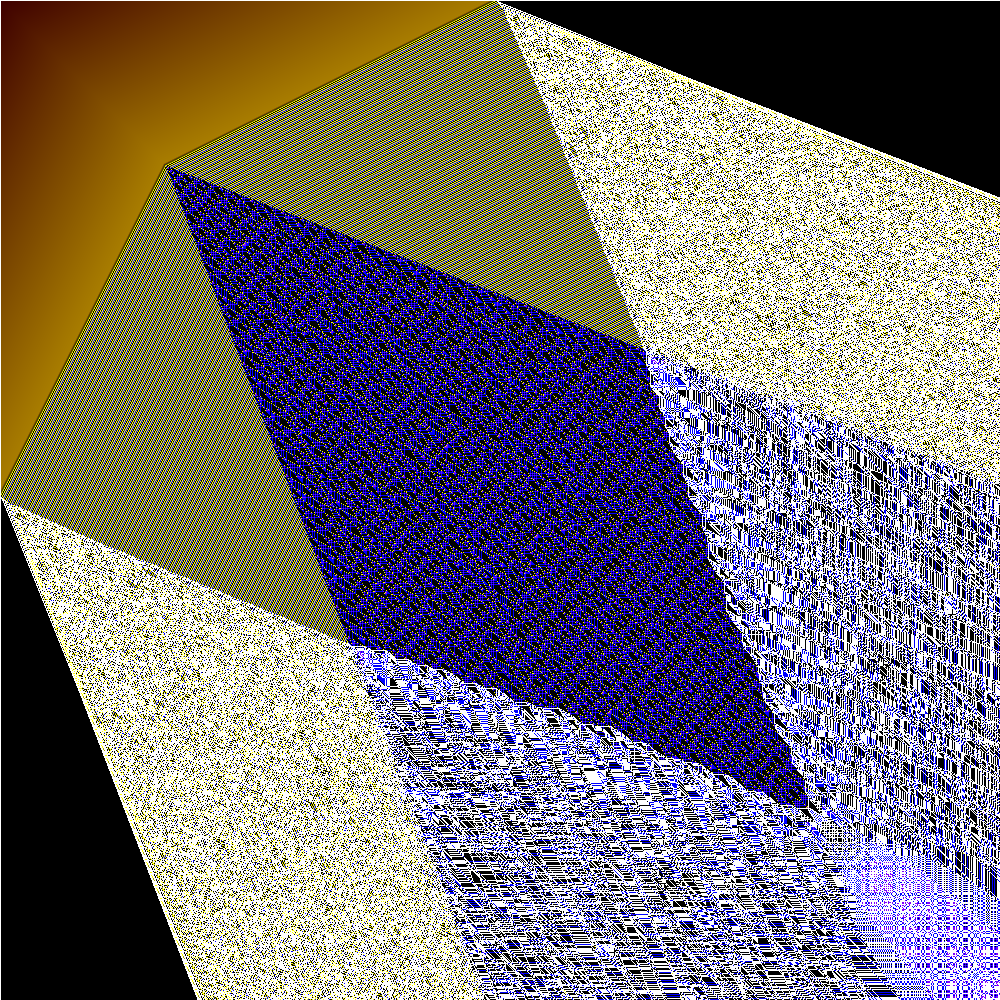}
~~~~~
\includegraphics[width=0.48\textwidth]{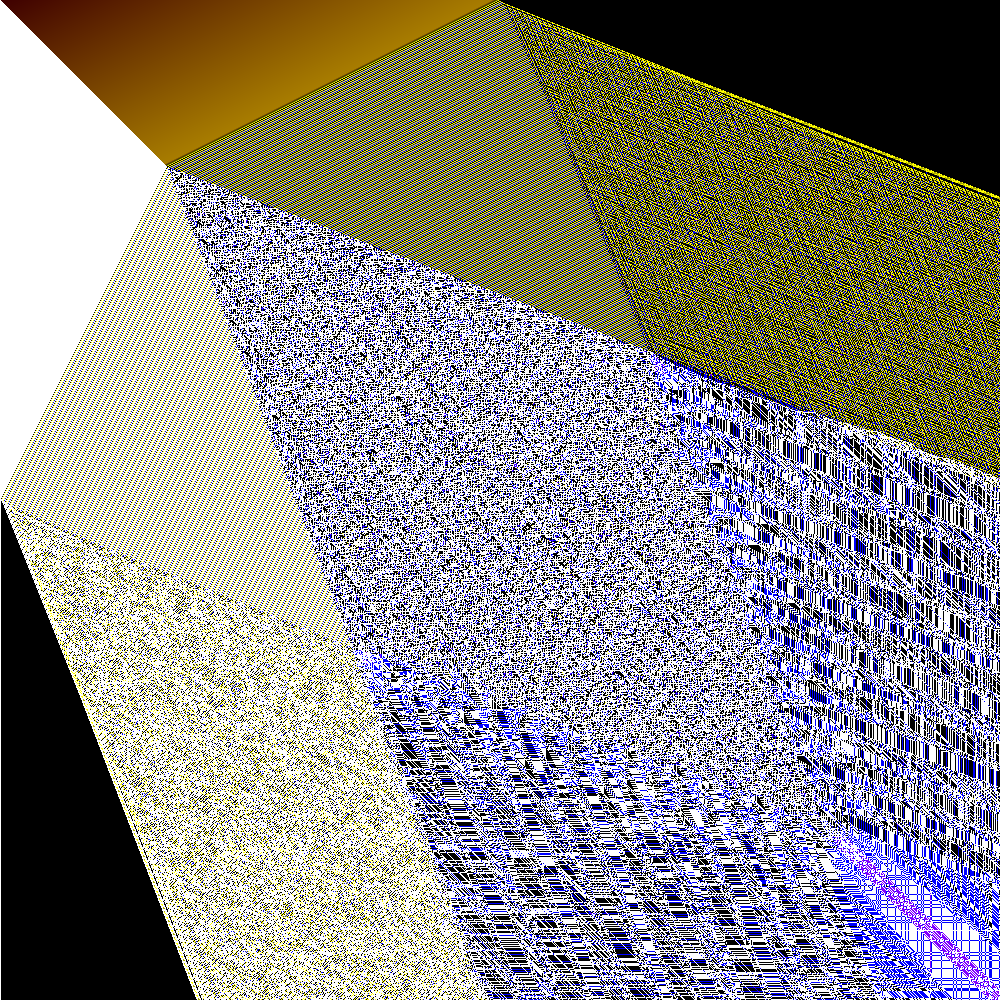}\hspace{-10pt}
\caption{\label{fig:rhombwhite}\label{fig:armwhite}
{\bf (left)}
A comparison between $k=497$ and $k=500$.
Positions where the two images differ are masked in white.
All other colors show places where the two images match,
meaning that if the palace number at position $(x,y)$ is $p$ in the image for $k=497$,
then the palace number at position $(x+1,y+1)$ is $p+3$ in the image for $k=500$.
P-positions are shown in yellow (and brown, in the hood),
and N-positions are shown in black and blue.
Note that the hood, \'epaulettes, and fabric match perfectly.
{\bf (right)}
A comparison between $k=499$ and $k=500$.
In this comparison, ``matching'' means that
if the palace number at position $(x,y)$ is $p$ in the image for $k=499$,
then the palace number at position $(x+1,y)$ is $p+1$ in the image for $k=500$.
Note that the \'epaulette and arm match perfectly, as does half of the hood.
}}
\end{figure}

The hood and the \'epaulets have a very regular structure.
The palace numbers in the hood increase steadily in each row and column,
increasing by one when the higher coordinate is incremented,
and by two when the lower coordinate is incremented.
Thus the hood contains all palace numbers from 0, at the upper left corner,
to $k$, where the hood meets the \'epaulettes at the casing,
a line with slope -1/2 (and -2) that connects the nose at $(k/3, k/3)$
to the shoulder at $(0, k)$ (and $(k, 0)$).
The hood is exactly the region where every move is a winning move,
because all possible further moves can be blocked,
thereby immediately winning the game.
When players are not making mistakes,
the game ends when (and only when) somebody moves into the hood,
thereby winning the game.

The palace numbers in the \'epaulettes form a two-dimensional periodic pattern,
with periods $(5,1)$ and $(4,3)$
(and all integral linear combinations of those base periods).
Only the palace numbers $k-1$, $k$, and $k+1$ appear in the \'epaulettes.
In the \'epaulette's periodic region of size 11,
$k-1$ (a P-position) appears 5 times,
and $k$ and $k+1$ (both N-positions ({\em no palace})) each appear 3 times.

The arms,
shown in the bottom row of Figure~\ref{fig:selforg}
have a random appearance,
although they often contain temporary black stripes
at one of the two angles parallel to their sides.
Despite this initial appearance of disorder,
the arms have many interesting properties,
discussed in Section~\ref{sec:arms}.

The fabric exhibits further self-organization.
The larger black regions visible in the middle row of Figure~\ref{fig:selforg}
form a rough grid,
and in much larger pictures ($k\gg1000$)
the grid morphs into a larger-scale grid,
which is at a slightly different angle
and has cells about 3.5 times larger.
Regions of the small-grid pattern appear to
travel through the large grid like meta-gliders,
visible in the top right of Figure~\ref{fig:gliders}.
These grid cells are separated by threads of P-positions,
which are able to split and merge
to form smaller and larger threads,
and sometimes seem to disappear.

Threads typically have a measurable (vertical, horizontal, and/or diagonal) thickness,
which is added when they merge,
as happens frequently just after the fray,
as in Figure~\ref{fig:fray}.
For example,
in the left-hand warp 
the threads have integer vertical thicknesses,
that is,
each thread has a fixed number of P-positions that occur in every column.
Furthermore, this fixed number is always a Fibonacci number.

\begin{figure}[ht!]
\centering{
\includegraphics[
width=0.99\textwidth]{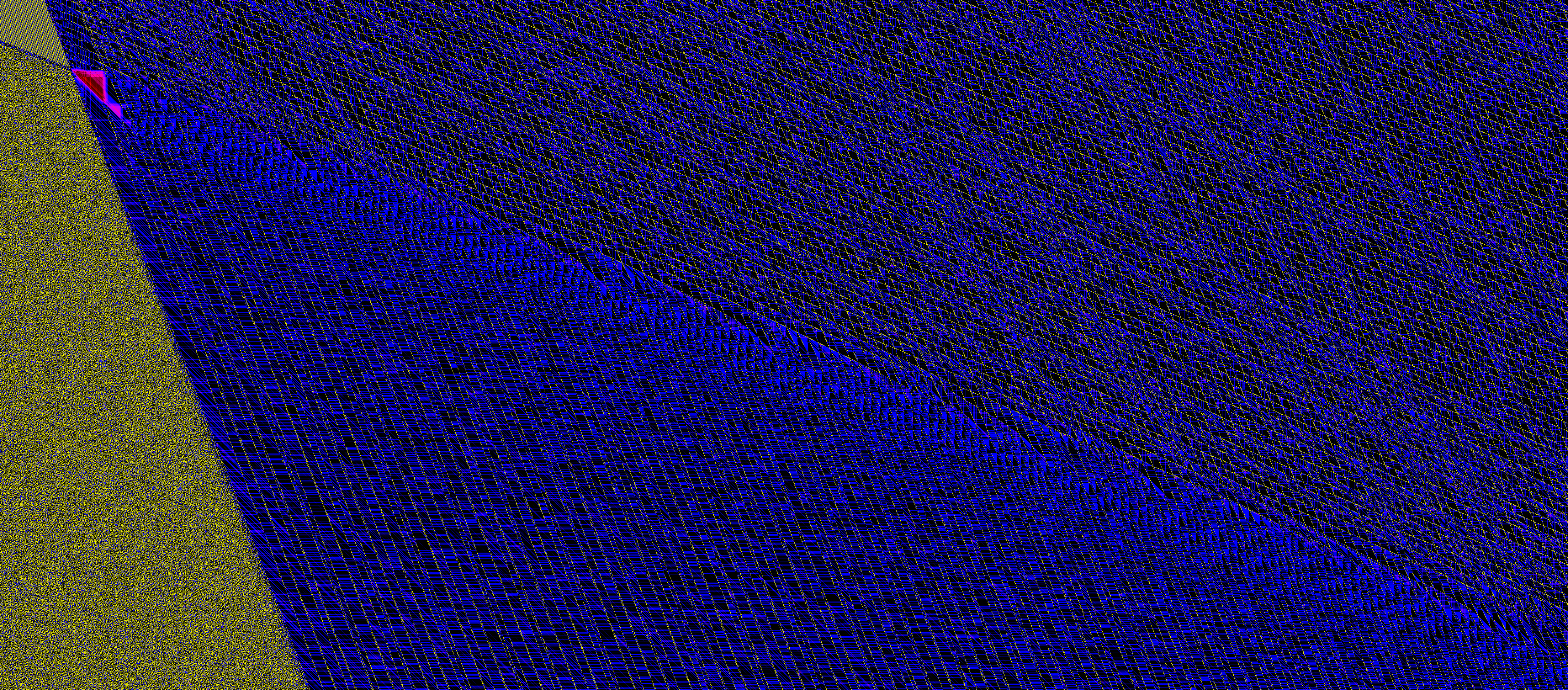}
\caption{\label{fig:armpit}\label{fig:gliders}\label{fig:fray}Game 9999.
The armpit in the upper left corner is generated when the rift hits the seam.
This starts the fray,
which separates the meta-glider behavior in the fabric
from the merging threads of the warp.
If viewing this document electronically, zoom in for detail.
}
}
\end{figure}

\begin{figure}[ht!]
\centering{
\includegraphics[
width=0.8\textwidth]{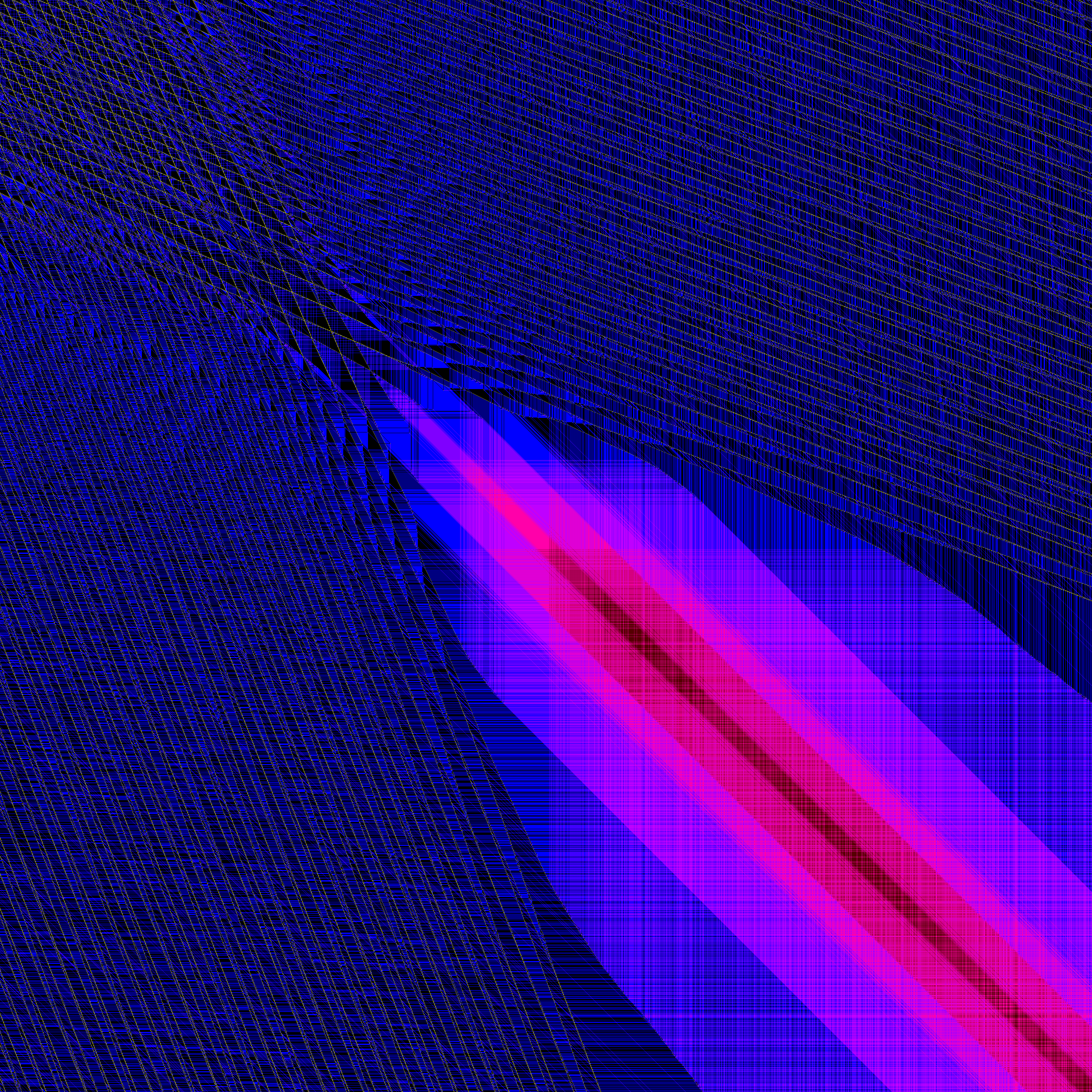}
\caption{Game 40003: A prism and its light beam.
The slightly higher palace numbers near the main diagonal
can give the impression of a beam of light being emitted by the prism.
When the two fraying edges of the fabric meet at the prism,
such higher palace numbers are often produced.
If viewing this document electronically, zoom in for detail.
}\label{fig:laser}
}
\end{figure}

\subsection{Region Prefix Properties}

\begin{figure}[ht!]
\centering{
\includegraphics[width=0.32\textwidth]{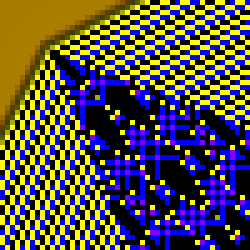}
\includegraphics[width=0.32\textwidth]{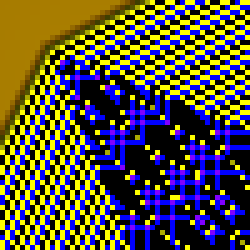}
\includegraphics[width=0.32\textwidth]{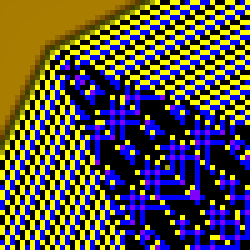}

\includegraphics[width=0.32\textwidth]{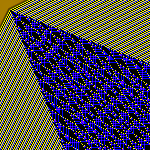}
\includegraphics[width=0.32\textwidth]{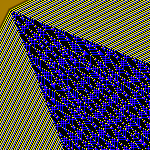}
\includegraphics[width=0.32\textwidth]{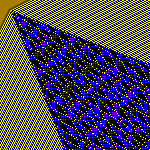}
\caption{
The 3 fabrics.
{\bf (left)}
$k=0 \bmod 3$.
{\bf (center)}
$k=1 \bmod 3$.
{\bf (right)}
$k=2 \bmod 3$.
The upper row is the same as the lower row, but more zoomed in.
}\label{fig:rhomb}
}
\end{figure}

If we look around the nose,
we see one of three pictures,
depending on the value of $k \bmod 3$,
because this determines the precise pixel arrangement of the casing at the nose.
These three shapes are shown in Figure~\ref{fig:rhomb}.
Since there are only three possibilities for the hood boundary shape at the nose,
and the CA rule can then be used
to produce the \'epaulettes and fabric without knowing $k$,
we see that despite the chaotic nature of the fabric,
there are in fact only three fabric patterns that can be produced.
Figures~\ref{fig:selforg} and~\ref{fig:rhombwhite} both show examples of
how the full fabric area matches
between different games with $k$ congruent modulo 3.

Using the CA, starting from an infinite casing pattern,
we can make any of the three fabric patterns
in an infinitely large version.
The fabric patterns that we see in practice are simply
prefixes of one of these three infinite fabrics.

The arms similarly can be formed by the CA rule from the shoulder,
and in this case there is only one possible pattern.
Figure~\ref{fig:armwhite} shows how
this pattern remains very stable as $k$ increases.

\subsection{Properties of the Skin}\label{sec:arms}

Where the arm borders the outer space,
the arm grows a periodic skin (see Figure~\ref{fig:skinpattern}),
which is slightly separated from the rest of the arm,
except that it emits a vertical (in the case of the upper arm) line once per period.
This skin consists of solidly packed P-positions,
with a vertical thickness of $f_{2n}$ (or $f_{2n}+1$ where a line is emitted),
a diagonal thickness of $f_{2n+1}$,
a horizontal thickness of $f_{2n+2}$,
a horizontal period of $f_{2n+3}$,
and a vertical period of $f_{2n+1}$,
where $f_{2n}$ is the $2n^{\rm th}$ Fibonacci number.
The skin originally forms at the top of the arm with $n=1$,
and after about 200 pixels (horizontally, about 80 vertically)
it changes to the form with $n=2$,
then after about 7700 pixels it changes to the form with $n=3$.
We conjecture that for large $k$,
it will continue to thicken in this manner,
with $n$ increasing by one each time,
approaching a slope of $\phi^2$.
In the other direction, one can even see the $n=0$ stage of this pattern
as part of the \'epaulette at the start of the arm for the first 10 pixels or so,
although it is not clearly visible
due to the lack of black pixels between this skin and the rest of the arm.

\begin{figure}[ht!]
\centering{
\includegraphics[width=.8\textwidth]{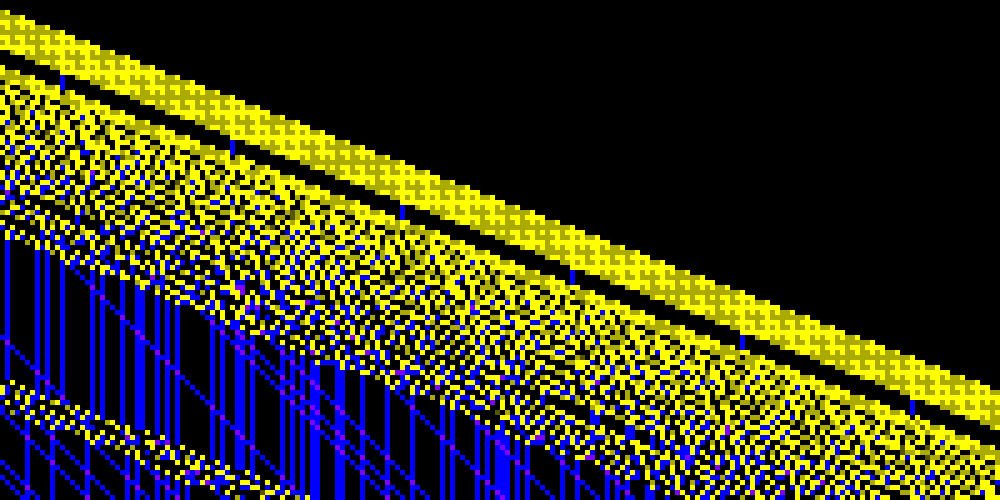}
\caption{Skin pattern.}
\label{fig:skinpattern}
}
\end{figure}

The boundary between the skin
(which has palace numbers $k-2$ and $k-1$)
and the outer space
(which has constant palace number $k$)
consists of steps of width 2 or 3\@.
Each time the skin thickens,
the pattern of steps expands according to the rule
$\{2\rightarrow23, 3\rightarrow233\}$,
starting from the pattern of just $2$ for the skin $n=0$.
The skin pattern of positions of palace numbers $k-2$ and $k-1$
can be computed from this pattern of steps and from the assumption that
there are $k-f_{2n+3}$ palaces to the left of the skin in each row.
This pattern is rotationally symmetric within each period
(between the columns that produce the vertical lines),
and with each thickening follows the expansion rule
(writing $-2$ for $k-2$, etc.)
$$ -2 \rightarrow
\left(\begin{array}{ccc}
-2 & ~ & -1 \\
-2 & ~ & -2
\end{array}\right),
\left(\begin{array}{ccccc}
-2 & ~ & -2 & ~ & -1 \\
-2 & ~ & -2 & ~ & -2
\end{array}\right),
\left(\begin{array}{ccc}
-1 & ~ & -1 \\
-2 & ~ & -1 \\
-2 & ~ & -2
\end{array}\right),
\left(\begin{array}{ccccc}
-2 & ~ & -2 & ~ & -1 \\
-2 & ~ & -2 & ~ & -2 \\
-2 & ~ & -1 & ~ & -1
\end{array}\right),
$$
$$ -1 \rightarrow
\left(\begin{array}{ccc}
-1 & ~ & -1 \\
-2 & ~ & -1
\end{array}\right),
\left(\begin{array}{ccccc}
-2 & ~ & -2 & ~ & -1 \\
-2 & ~ & -1 & ~ & -1
\end{array}\right),
\left(\begin{array}{ccc}
-1 & ~ & -1 \\
-2 & ~ & -1 \\
-1 & ~ & -1
\end{array}\right),
\left(\begin{array}{ccccc}
-2 & ~ & -2 & ~ & -1 \\
-1 & ~ & -1 & ~ & -1 \\
-2 & ~ & -1 & ~ & -1
\end{array}\right),
$$
where one of the four expansion matrices is chosen based on the required size.
The widths are determined by the step pattern,
and the heights are partially defined by
the property that all matrices produced by a given row
should have a row in common.


\section{Conjectures and Questions}
Most of the observations in Section~\ref{sec:selforg} may be viewed as open problems. Here we list of a few.
\begin{itemize}
\item
Is it the case that a
sufficiently thick arm will grow
thicker and thicker skin over time?
\item
The arm's skin (as observed for $n=1$ and $n=2$) has
a vertical thickness of $f_{2n}$ (except where a line is emitted),
a diagonal thickness of $f_{2n+1}$,
a horizontal thickness of $f_{2n+2}$,
a horizontal period of $f_{2n+3}$,
and a vertical period of $f_{2n+1}$,
where $n$ is the thickness level
(here $f_i$ is the $i^\textrm{th}$ Fibonacci number).
Can this pattern be explained, and does it continue?
\item
Are the inner sectors beyond the fabric essentially different for different~$k$,
starting with distinct armpits (see Figure~\ref{fig:gliders}),
and eventually producing ever different and irregular innermost P-threads?
\item Are the armpits the unique regions from which the queen views the most palaces (for large $k$)?
\item
Do the innermost P-threads always have slopes corresponding to algebraic numbers?
If there is only one innermost upper P-thread,
is the slope a root of a second degree polynomial with rational coefficients?
\item
How many non-periodic threads can there be for a given $k$?
We conjecture at most two upper ones (as occurs for $k=46$),
which must be the inner ones;
with all others eventually becoming periodic.
\item
Why is it that the threads of the warp merge in such a way
that their thickness is always a Fibonacci number?
\item Is it true that the threads (also inside the arm and in particular in the inner sector) bound the number of palace positions that the queen sees looking just diagonally, and this number is mostly just one number (otherwise at most one out of two numbers)? 
\end{itemize}



\end{document}